\documentclass{amsart}
\usepackage{amssymb,amsmath, amsthm,latexsym}
\usepackage{graphics}
\usepackage{psfrag}
\usepackage{amscd}
\usepackage{graphicx}
\newcommand{\cal}[1]{\mathcal{#1}}
\theoremstyle{plain}
\newtheorem{theo}{Theorem}

\newtheorem{lemma}{Lemma}[section]
\newtheorem{theorem}[lemma]{Theorem}  
\newtheorem{proposition}[lemma]{Proposition}
\newtheorem{corollary}[lemma]{Corollary}
\theoremstyle{definition}
\newtheorem*{defi}{Definition}
\newtheorem{definition}[lemma]{Definition}

\let\egthree=\phi
\let\phi=\varphi
\let\varphi=\egthree

\begin{document}
\title[Asymptotic dimension and disk graphs I]
{Asymptotic dimension and the disk graph I}
\author{Ursula Hamenst\"adt}
\thanks{Partially supported by 
ERC Grant "Moduli"}
\date{October 7, 2018}


\begin{abstract}
For a 3-manifold $M$ and a subsurface $X$ of the boundary of $M$
with empty or incompressible boundary 
we use surgery to identify a graph whose
vertices are disks with boundary in $X$ and 
which is quasi-isometrically embedded in the curve
graph of $X$. \end{abstract}

\maketitle


\section{Introduction}

Consider an oriented 3-manifold $M$ and a subsurface $X$ of the 
boundary of $M$. We require that the boundary of $X$ either is empty,
or it is incompressible in $M$.
We also require that the Euler characteristic of $X$
is negative. The prototypical example is a \emph{handlebody} of genus
$g\geq 2$, i.e.
a compact three-dimensional manifold which can
be realized as a closed regular neighborhood in $\mathbb{R}^3$
of an embedded bouquet of $g$ circles. Its boundary
is a closed oriented surface of genus $g$. 

The \emph{disk graph} ${\cal D\cal G}(X)$ of $(M,X)$ is the metric graph whose 
vertices are isotopy classes of properly embedded disks in 
$M$ with boundary in $X$ and 
where two such disks are connected by an edge of length one
if they can be realized disjointly. Assigning to a disk its boundary
then defines an embedding of the disk graph into the 
\emph{curve graph} ${\cal C\cal G}(X)$ of $X$. 

The curve graph ${\cal C\cal G}(X)$ is a locally infinite 
geodesic metric graph which is hyperbolic in the sense of Gromov \cite{MM99}. 
The disk graph ${\cal D\cal G}(X)$ is a \emph{quasi-convex} subset of 
${\cal C\cal G}(X)$ \cite{MM04}. This means that there exists a number
$c>1$ with the following property. For any two points 
$x,y\in {\cal D\cal G}(X)$, there exists a path in ${\cal D\cal G}(X)$ which 
connects $x$ to $y$ and which is contained in the $c$-neighborhood of
a geodesic in ${\cal C\cal G}(X)$ connecting $x$ to $y$. 

As ${\cal C\cal G}(X)$ is a locally infinite, this does not imply that 
the inclusion ${\cal D\cal G}(X)\to {\cal C\cal G}(X)$ is 
a quasi-isometric embedding. Indeed, as was discovered by
Masur and Schleimer \cite{MS13}, this is not the case. 
Namely, in the terminology of their paper, the disk graph 
of a handlebody of genus $g\geq 2$ has
\emph{holes} consisting of convex subsets of infinite diameter
whose images in ${\cal C\cal G}(X)$ have uniformly bounded diameter.
Nevertheless, the main result of 
\cite{MS13} shows that the disk graph of the handlebody is hyperbolic.
Furthermore, somewhat indirectly, Masur and Schleimer  
describe how to "fill" the
holes and, by adding edges to ${\cal D\cal G}(X)$,
to construct a graph whose vertices are disks and 
which is quasi-isometrically embedded in the curve graph.

The main purpose of this article is to define such a graph explicitly and
to give a purely combinatorial proof that it embeds
quasi-isometrically into the curve graph. This construction
is used in \cite{H16} to give an alternative proof of hyperbolicity of 
the disk graph and zo determine its Gromov boundary. In \cite{H17}
we use the finer structure of the disk graph established along the way 
to show that its asymptotic dimension is finite.

A  construction which is closer to the viewpoint of
Masur and Schlei\-mer is due to Ma. In the article
\cite{Ma14} one also finds an interpretation of some of
our results using the viewpoint of 
of Masur and Schleimer.

To introduce the graph we are interested in, call
a simple closed curve 
$c$ on $X$
\emph{diskbusting} if $c$ has an essential 
intersection with the boundary of every disk.

Define an \emph{$I$-bundle generator} for $X$ to be
a diskbusting simple closed
curve $c$ on $X$ with the following property.
There is a compact surface $F$ with 
a distinguished boundary component $\alpha\in \partial F$,
and there is a homeomorphism of the orientable $I$-bundle
${\cal J}(F)$ 
over $F$ into $M$ which maps $\alpha$ to $c$
and which maps the union of the horizontal boundary of 
$F$ with the $I$-bundle over $\alpha$ onto the 
complement in $X$ of a tubular neighborhood of the boundary of $X$.


\begin{defi}\label{super}
The \emph{super-conducting disk graph} 
is the graph ${\cal S\cal D\cal G}(X)$ 
whose vertices are isotopy classes of essential 
disks with boundary in $X$ and where two vertices 
$D_1,D_2$ are connected by an edge of length one 
if and only if one of the following two possibilities holds.
\begin{enumerate}
\item There is an essential simple closed curve on $X$ 
which can be realized disjointly from both
$\partial D_1,\partial D_2$.
\item There is an $I$-bundle generator $c$ for $X$ which
intersects both $\partial D_1,\partial D_2$ in 
precisely two points.
\end{enumerate}
\end{defi}

Since the distance in the curve graph ${\cal C\cal G}(X)$ of $X$
between two simple closed curves which intersect in two points 
does not exceed  $3$ \cite{MM99}, 
the natural vertex inclusion extends to a coarse $6$-Lipschitz map
${\cal S\cal D\cal G}(X)\to {\cal C\cal G}(X)$. We show

\begin{theo}\label{hyperbolic}
The natural vertex inclusion extends to a quasi-isometric embedding 
${\cal S\cal D\cal G}(X)\to {\cal C\cal G}(X)$.
\end{theo}

The constants for the quasi-isometric
embeddings are bounded from above by 
an explicit quadratic polynomial in the Euler characteristic 
of $X$.


The requirement that the boundary of the surface $X$ is incompressible in $M$
is essential for Theorem \ref{hyperbolic}. In the statement of the following 
result, we tacitly assume that the graph ${\cal D\cal G}(X)$ is not trivial.

\begin{theo}\label{spot}
If $X$ is a subsurface of the boundary of $M$ of genus $g\geq 2$, 
with a single compressible boundary component, then 
the graph ${\cal D\cal G}(X)$ 
is not a quasi-convex subset of the curve graph.  
\end{theo}


\bigskip
\noindent
{\bf Organization:} In Section \ref{distance} we
use surgery of disks to relate the distance 
in the superconducting disk graph
to intersection numbers of boundary curves.
 
In Section \ref{distanceinthe} we give an 
effective estimate of the distance in the curve graph using
train tracks. The results in this section are independent of
the rest of the article. 

Together with a 
construction of \cite{MM04}, this 
is used in 
Section \ref{quasigeodesics} to show Theorem \ref{hyperbolic}. 
In Section \ref{gromov} we
identify the Gromov boundary of ${\cal S\cal D\cal G}(X)$
in the case $M$ is a handlebody of genus $g\geq 2$ and 
$X$ is its boundary surface. 
The proof of Theorem \ref{spot} is contained in 
Section \ref{disksforahandle}.

\bigskip

\noindent 
{\bf Acknowledgement:} I am indebted to Saul Schleimer for 
making me aware of a missing case 
in the surgery argument in Section \ref{distance} in 
a first draft of this paper and for 
sharing his insight in the disk graph with me. 
The results in Sections 2-4 of this article were obtained 
in summer 2010 while I
visited the University of California in Berkeley.
I am especially grateful to an anonymous referee who
suggested a considerable simpliciation of the proof of 
Lemma \ref{twocomponent} and for other useful comments,
including pointing out the reference \cite{Ma14}.

\section{Distance and intersection}\label{distance}

In this section we consider an
arbitrary oriented 3-manifold $M$ together with a compact oriented
subsurface $X$ of the boundary of $M$. The surface $X$ may have boundary
$\partial X$, but any boundary component of $X$ is supposed to 
be incompressible in $M$.

By a \emph{disk} we always mean an embedded 
essential disk in $M$
with boundary in $X$. 
As the boundary of $X$ is not diskbounding by assumption,
the boundary
$\partial D$ of such a disk is an essential curve in $X$. 
Two disks $D_1,D_2$ are in 
\emph{normal position} if their boundary circles
intersect in the minimal number of points and if 
every component of $D_1\cap
D_2$ is an embedded arc
in $D_1\cap D_2$ 
with endpoints in $\partial D_1
\cap \partial D_2$.
In the sequel we always assume that disks are in normal position;
this can be achieved by modifying one of the two disks with an
isotopy.

Let $D$ be any disk and let $E$ be a
disk which is not disjoint from $D$. 
A component $\alpha$ of $\partial E-D$ 
is called an \emph{outer arc} of $\partial E$ relative to $D$
if there is a component
$E^\prime$ of $E-D$ whose boundary is composed of $\alpha$
and an arc $\beta\subset D$. The interior of $\beta$
is contained in the interior of $D$. We call such a disk $E^\prime$
an \emph{outer component} of $E-D$. An outer component
of $E-D$ intersects $X$ in an outer arc $\alpha$
relative to $D$, and 
$\alpha$ 
intersects $\partial D$ in opposite directions
at its endpoints. 

For every disk $E$ which is not disjoint
from $D$ there are 
at least two distinct outer components $E^\prime,E^{\prime\prime}$
of $E-D$.
There may also be components of $\partial E-D$ which
leave and return to the same side of $D$ but which are
not outer arcs. An example of such a component 
is a  subarc of $\partial E$ which
is contained in the boundary of 
a rectangle component of $E-D$
leaving and returning to the same side of $D$. The boundary
of such a rectangle consists of 
two subarcs of $\partial E$ with endpoints 
on $\partial D$ which 
are homotopic  relative to $\partial D$, and two 
arcs contained in $D$.

\begin{figure}
\begin{center}
\psfrag{alpha}{$\alpha$}
\includegraphics[width=0.7\textwidth]{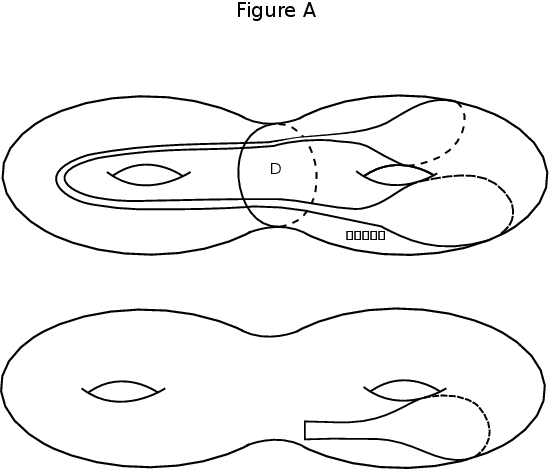}
\end{center}
\end{figure}

Let $E^\prime\subset E$ be an outer component of $E-D$ 
whose boundary is composed of an outer arc $\alpha$ and a 
subarc $\beta=E^\prime\cap D$ of $D$. 
The arc $\beta$
 decomposes the disk $D$ into two half-disks 
$P_1,P_2$. The unions
$Q_1=E^\prime\cup P_1$ and $Q_2=E^\prime\cup P_2$ are embedded disks in 
$M$ which up to isotopy are disjoint and disjoint from $D$.
For $i=1,2$ we say that the disk $Q_i$ is obtained from $D$ by 
\emph{simple surgery} at the outer component
$E^\prime$ of $E-D$ (see e.g. \cite{S00} for 
this construction). Since $D,E$ are in minimal position, 
the disks $Q_1,Q_2$ are essential.


In the introduction we defined two 
graphs of disks with boundary in $X$. 
We called them 
\emph{disk graph} ${\cal D\cal G}(X)$
and \emph{superconducting disk graph}
${\cal S\cal D\cal G}(X)$, respectively. 

Each disk in $M$ with boundary in $X$
can be viewed as a vertex in the 
disk graph ${\cal D\cal G}(X)$
and the superconducting
disk graph ${\cal S\cal D\cal G}(X)$ of $X$.
We will work with both graphs simultaneously.
Denote by $d_{\cal D}$ (or
$d_{\cal S}$) the distance in ${\cal D\cal G}(X)$ (or in 
${\cal E\cal D\cal G}(X)$ or in ${\cal S\cal D\cal G}(X)$). Note that for
any two disks $D,E$ we have
\[d_{\cal S}(D,E)\leq d_{\cal D}(D,E).\]

In the sequel we always assume that all curves and multicurves
on $X$ are essential. 
For two simple closed multicurves $c,d$ on $X$ let 
$\iota(c,d)$ be the geometric 
intersection number between $c,d$. 
The following lemma \cite{MM04} implies that the
graph ${\cal D\cal G}(X)$ is connected.
We provide the short proof for completeness.

\begin{lemma}\label{fellowtravel}
Let $D,E\subset M$ be any two
disks with boundary in $X$. Then $D$
can be connected to a disk $E^\prime$ 
with boundary in $X$ which is disjoint from 
$E$ by at most 
$\log_2(\iota(\partial D,\partial E)/2)$ simple surgeries.
In particular, 
\[d_{{\cal D}}(D,E)\leq \log_2(\iota(\partial D,\partial E)/2)+1.\]
\end{lemma}
\begin{proof} Let $D,E$ be two disks in normal position, with boundary
in $X$. Assume that 
$D,E$ are not disjoint. Then 
there is an outer component $E^\prime$ of $E-D$. 
The endpoints of the outer arc $\partial E^\prime\cap X$
decompose $\partial D$ into two arcs
$\beta_1,\beta_2$. Choose the arc with fewer intersections with
$\partial E$, say the arc $\beta_1$. 
The disk $D^\prime$  obtained by simple surgery
of $D$ at this component which contains $\beta_1$ in its boundary 
is essential, with boundary in $X$. 
Moreover, 
$D^\prime$ 
is disjoint from $D$, i.e. we have 
$d_{{\cal D}}(D^\prime,D)=1$,  and 
\begin{equation}\label{intersection}  
\iota(\partial E, \partial D^\prime)\leq 
\iota(\partial D,\partial E)/2.\end{equation}
The lemma now follows by induction on 
$\iota(\partial D,\partial E)$.
\end{proof}

Consider an oriented $I$-bundle ${\cal J}(F)$ over a compact
(not necessarily oriented) surface $F$ with (not necessarily connected)
boundary $\partial F$. The boundary $\partial {\cal J}(F)$ decomposes into 
the \emph{horizontal boundary} and the \emph{vertical boundary}. The vertical
boundary is the interior of the restriction of the $I$-bundle to $\partial F$ and 
consists of a collection of pairwise disjoint open incompressible annuli. 
The horizontal boundary is the complement of the vertical boundary
in $\partial {\cal J}(F)$. 

For a given boundary component $\alpha$ of $F$, the union of the horizontal
boundary of ${\cal J}(F)$ with the $I$-bundle over $\alpha$ is a compact
connected orientable surface $F_\alpha\subset \partial {\cal J}(F)$. 
The boundary of $F_\alpha$ is empty if and only if the boundary of $F$ is 
connected. If the boundary of $F$ is not connected then $F_\alpha$ is 
properly contained in the boundary $\partial {\cal J}(F)$ of ${\cal J}(F)$. 
The complement $\partial {\cal J}(F)-F_\alpha$ is a union of incompressible annuli.

\begin{definition}\label{diskbustingibundle}
An \emph{$I$-bundle generator} is an essential simple
closed curve $\gamma\subset X$
with the following property. There is a compact surface $F$ with 
non-empty boundary $\partial F$, there is a boundary component 
$\alpha$ of $\partial F$, and there is 
an orientation preserving  embedding $\Psi$ of the oriented 
$I$-bundle ${\cal J}(F)$ over $F$ into $M$ which maps $\alpha$ to $\gamma$
and which maps  $F_\alpha$ onto the complement in $X$ 
of a tubular neighborhood of 
the boundary $\partial X$ of $X$. 
\end{definition}
We call the surface $F$ the \emph{base} of the $I$-bundle generated by $\gamma$.
Note that an $I$-bundle generator $\gamma$ is 
\emph{diskbusting}, i.e it intersects 
every disk in $M$ with boundary in $X$.

If $\gamma$ is a separating $I$-bundle generator in $X$ 
with base surface $F$ (i.e. $\gamma\subset X$ is a separating
simple closed curve which also is an $I$-bundle generator) 
then $F$ is orientable and the genus $g$ of $X$ is
even. 
Moreover, the $I$-bundle ${\cal J}(F)=F\times [0,1]$ is trivial.
The $I$-bundle over every essential arc in $F$
with endpoints in $\partial F$ is an embedded disk in $M$.
If $\gamma$ is a non-separating $I$-bundle generator in 
$X$ then the base $F$ of the $I$-bundle is 
non-orientable.

For each $I$-bundle ${\cal J}(F)$, 
there is an orientation reversing involution
$\Phi:{\cal J}(F)\to {\cal J}(F)$ which 
acts as
a reflection in the fiber. Up to isotopy, the $I$-bundle over
any essential arc $\beta$ on the surface $F$ 
with endpoints in the same boundary component 
$\alpha$ is a $\Phi$-invariant disk which 
intersects $\alpha$ in precisely two points.

If $D,E\subset M$ are disks in normal position 
then each component of $D-E$ is a disk.
Furthermore, each component of $D\cap E$ 
is a properly embedded arc in $D$ which
decomposes $D$ into two connected components. 
Therefore the graph
dual to the cell decomposition of $D$ whose two-cells are the
components of $D-E$ is a tree. If
$D-E$ only has two outer components then 
this tree is just a line segment.
The following lemma analyzes the case that this holds true for both
$D-E$ and $E-D$. For its formulation, we say that 
two simple closed curves $c,d$ \emph{fill} the surface $X$
if $c,d$ are contained in $X$ and if there is no
essential simple closed curve in $X$ which is disjoint from 
$c\cup d$. 

The proof of the following lemma uses a suggestion of a referee
which lead to a considerable simpliciation of the argument.

\begin{lemma}\label{twocomponent}
Let $D,E\subset M$ be disks in normal position 
with boundary in the surface $X$. 
If $D-E$ and $E-D$ only have two outer components 
and if $\partial D, \partial E$ fill $X$ 
then
there exists an $I$-bundle ${\cal J}(F)$ over a compact surface
$F$ and an embedding $\Psi:{\cal J}(F)\to M$ with the
following properties. There is a boundary component 
$\alpha$ of $F$ such that $\Psi(\alpha)$ is an $I$-bundle generator 
in $X$, and $D,E$ are the images under $\Psi$ of 
$I$-bundles over embedded arcs  
$\delta,\beta$ in $F$ with endpoints on $\alpha$. 
\end{lemma}
\begin{proof} Let $D,E$ be two disks in normal position, with boundary in 
the surface $X$. 
Assume that 
$D-E$ and $E-D$ only have two outer components.
Then each component of $D-E,E-D$ either is an
outer component or a rectangle, i.e. a disk whose boundary
consists of two 
components of $D\cap E$ and 
two arcs contained in $\partial D\subset X$ or 
$\partial E\subset X$, respectively. 
Assume 
that $\partial D, \partial E$ fill up $X$.
This means that $\partial X-({\partial D\cup \partial E})$ is 
a union of disks
and peripheral annuli about the boundary components of $X$. 

Choose tubular
neighborhoods $N(D),N(E)$ of $D,E$ in $M$ which are 
homeomorphic to an interval bundle over a disk and which 
intersect $X$ in an embedded
annulus. 
We may assume that the interiors 
$A(D),A(E)$ of these annuli are contained in 
the interior of $X$.
Then $\partial N(D)-A(D)$,
$\partial N(E)-A(E)$ is the union of two properly 
embedded disjoint disks in $M$ 
isotopic to $D,E$.
We may assume that $\partial N(D)-A(D)$ 
is in normal position with
respect to $\partial N(E)-A(E)$ and that  
\[S=\partial (N(D)\cup N(E))-(A(D)\cup A(E))\] is a 
compact surface with boundary which is properly
embedded in $M$. Since $M$ is assumed
to be oriented, the boundary $\partial (N(D)\cup N(E))$ 
of $N(D)\cup N(E)$ 
has an induced orientation which restricts to an 
orientation of $S$. 

Now note that $N(D)\cup N(E)$ has the structure of an interval
bundle over a surface with the property that 
each intersection component of $D\cap E$ is a fibre of this bundle.
Namely, for each outer component $C$ of $D-E$ or $E-D$ choose
two points in the interior of $\partial C\cap X$ so that 
the boundary of $C$ can be viewed as a rectangle, with one side 
$\rho$ the
component of $D\cap E$ contained in the boundary of $C$.
Foliate this rectangle in standard way by intervals so that
$\rho$ is a leaf of this foliation. Similarly, each component of 
$D-E$ or $E-D$ which is not an outer components 
contains two components of $D\cap E$ in its boundary, and it 
can be foliated into
intervals in such a way that the two components of 
$D\cap E$ in its boundary are leaves. This foliation of 
$D\cup E$ can naturally be extended to a foliation of 
$N(D)\cap N(E)$ by intervals. 
With the exception of a subarc of the boundary of an outer component, 
the leaves of this foliation intersect the boundary surface $X$ only
at their endpoints. 

By assumption, 
$\partial D\cup \partial E$ decompose $X$ into a 
union of polygons, i.e. disks bounded by finitely many
subarcs of $\partial D\cup \partial E$ and peripheral annuli. Such a polygon
$P$ is contained in the boundary of a component $V$ of
$M-(D\cup E)$. 
The intersection $\partial V\cap X$ has two connected components.
One of these components is the polygon $P$, the other component
$P^\prime$ 
either is a polygon component of $X-(\partial D\cup \partial E)$, or
it contains a boundary component of $X$.

The complement of $P\cup P^\prime$ in $\partial V$ is a finite collection
$W$ of fibred rectangles. The base of such a rectangle is an edge in 
the boundary $\partial P$ of $P$. The side of the rectangle opposite
to the base is an arc in the boundary of $P^\prime$. Since $P$ is 
a topological disk, this implies that the same holds true for 
$P^\prime$ and $V$ is a 3-ball.

As a consequence, each component $V$ of $H-(D\cup E)$ which contains
a polygonal component of $X-(\partial D\cup \partial E)$ in its
boundary is a ball whose boundary consists of $P$, a finite union
${\cal R}$ of fibred rectangles with pase $\partial P$ and a second
polygonal component $P^\prime$ of $X-(\partial D\cup \partial E)$.
The $I$-bundle structure on ${\cal R}$ naturally extends to 
an $I$-bundle structure on $V$. Therefore the union of 
$N(D)\cup N(E)$ with these components is an $I$-bundle whose 
boundary contains the complement of a small neighborhood of 
the boundary of $X$. The involution of the $I$-bundle preserves each
component $V$ of $M-(D\cup E)$ determined by a polygon in 
$X-(\partial D\cup \partial E)$, and it exchanges the 
two components of $V\cap X$.

Now note that by construction, $\partial D,\partial E$ intersects 
the fixed point set of the involution only at two points, and these
two points are contained in the interiors of the unique fibres of the
bundle which are subarcs of the two outer components of 
$D-E,E-D$, respectively. As the intersection of this
fixed point set with $X$ 
is the generator $\gamma$ of the $I$-bundle in the
sense defined above,  $\gamma$ has
all the properties stated in the Lemma. This completes the proof.
\end{proof}

We use Lemma \ref{twocomponent} to show 

\begin{proposition}\label{distanceinter2}
Let $D,E\subset M$ be essential disks with boundary in $X$.
If there is an essential simple closed curve $\alpha\subset X$ 
which intersects $\partial D,\partial E$ in at most 
$k\geq 1$ points 
then $d_{\cal S}(D,E)\leq 2k+4$. 
\end{proposition}
\begin{proof}
Let $D,E$ be essential disks in normal position as in the proposition
which are not disjoint.

Let $\alpha$ be an essential simple closed curve 
in $X$ which intersects both
$\partial D$ and $\partial E$ in at most $k\geq 1$ points.
We may assume that these intersection points are disjoint from
$\partial D\cap \partial E$. 

Let $p\geq 2$ (or $q\geq 2$) be the number
of outer components of $D-E$ (or of $E-D$).
If $p=2,q=2$ then 
Lemma \ref{twocomponent} shows that either
$d_{{\cal S}}(D,E)\leq 1$ (in the case that $\partial D,\partial E$ do not
fill up $X$) or 
$\partial D,\partial E$ intersect some 
$I$-bundle generator $\gamma\subset X$ in precisely two points, and we have 
$d_{\cal S}(D,E)=1$. 

Let $j\leq k,j^\prime\leq k$ be the number of intersection
points of $D,E$ with $\alpha$. As $d_{\cal S}(D,E)=1$ if
$j=j^\prime=0$ we may assume that $j+j^\prime\geq 1$. 
Thus it suffices to show the following. 
If $\max\{p,q\}\geq 3$
then there is a simple surgery 
transforming the pair $(D,E)$ to a pair $(D^\prime,E^\prime)$
with the following properties.
\begin{enumerate}
\item $D^\prime$ is disjoint from $D$, $E^\prime$ is disjoint from $E$.
\item Either $D=D^\prime$ or $E=E^\prime$.
\item The total number of intersections of $\alpha$ with 
$D^\prime\cup E^\prime$ is strictly smaller than $j+j^\prime$.
\end{enumerate}

To this end assume without loss of generality
that $q\geq 3$. If $j/2>j^\prime/3$ then 
choose an outer component $E_1$ of $E-D$ with at most
$j^\prime/3$ intersections with $\alpha$. This is possible
because $E-D$ has at least three outer components.
Let $D_1$ be the component of $D-E_1$ which 
intersects $\alpha$ in at most $j/2$ points.
Replace $D$ be the disk $D^\prime=D_1\cup E_1$ which is disjoint from $D$ and
has at most $j/2+j^\prime/3< j$ intersections with $\alpha$.

On the other hand, if $j/2\leq j^\prime/3$ then 
choose an outer component $D_1$ of $D-E$ with at most
$j/2$ intersections with $\alpha$. Let $E_1$ be the component of 
$E-D_1$ with at most $j^\prime/2$ intersections with
$\alpha$ and replace $E$ by the disk 
$E^\prime=E_1\cup D_1$ which is disjoint from 
$E$ and intersects $\alpha$ in at most
$j/2+j^\prime/2< j^\prime$ points. 

This is what we wanted to show. 
\end{proof}

{\bf Remark:} The arguments in this section use the fact
that every simple surgery of a disk at an outer component of another
disk yields an essential disk in $M$.
They are not valid for surfaces $X\subset \partial M$ with compressible
boundary.

\section{Distance in the curve graph}\label{distanceinthe}

The purpose of this section is to establish an 
estimate for  the distance in the curve graph 
of a compact oriented surface $X$ 
of genus $g\geq 0$ with 
$m\geq 0$ boundary components 
and $3g-3+m\geq 2$. This estimate
which will be
essential for a geometric description of the
superconducting disk graph. 

The curve graph of a compact oriented 
surface $X$ with boundary coincides with the curve graph
obtained from $X$ by replacing each boundary component by
a puncture. 
As considering surfaces with punctures rather than
bordered surfaces has advantages for our exposition,
we consider in the remainder of this section an arbitrary
closed oriented
surface $S$ from which 
a finite set of points  
have been deleted. This results in this section are independent 
from the rest of the paper.

The idea is 
to use 
\emph{train tracks} on $S$. 
We refer
to \cite{PH92} for all basic notions and constructions regarding
train tracks.

A train track $\eta$ (which may just be a 
simple closed curve) is \emph{carried}
by a train track $\tau$ 
if there is a map $F:S\to S$ of class $C^1$ 
which is homotopic to the identity, with $F(\eta)\subset \tau$ and 
such that the restriction of the differential $dF$ of $F$ to 
the tangent line of $\eta$ vanishes nowhere. Write
$\eta\prec\tau$ if $\eta$ is carried by $\tau$. If $\eta\prec \tau$ 
then the image of $\eta$ under a carrying map is a \emph{subtrack} 
of $\tau$ which does not depend on the choice of the carrying
map. Such a subtrack is a subgraph
of $\tau$ which is itself a train track. Write $\eta<\tau$ if 
$\eta$ is a subtrack of $\tau$.

A train track $\tau$ is called \emph{large} \cite{MM99} 
if each complementary component of $\tau$ is either
simply connected or a once punctured disk.
A simple closed curve $\eta$ carried by $\tau$ 
\emph{fills} $\tau$ if the image of $\eta$
under a carrying map is all of $\tau$.
A \emph{diagonal extension} of a large train track $\tau$
is a train track $\xi$ which can be obtained from $\tau$
by subdividing some complementary components which are
not trigons or once punctured monogons.

A \emph{trainpath} on $\tau$ is an immersion
$\rho:[k,\ell]\to\tau$ which maps every 
interval $[m,m+1]$ diffeomorphically onto a branch of $\tau$.
We say that $\rho$ is \emph{periodic} if $\rho(k)=\rho(\ell)$
and if the inward pointing tangent of $\rho$ at $\rho(k)$
equals the outward pointing tangent of $\rho$ at
$\rho(\ell)$.
Any simple closed curve carried by a train track $\tau$ defines
a periodic trainpath and a 
\emph{transverse measure} on $\tau$. The space of transverse
measures on $\tau$ is a cone in a finite dimensional real vector space.
Each of its extreme rays is spanned by 
a \emph{vertex cycle} which is a simple closed curve carried
by $\tau$.
A vertex cycle defines a periodic trainpath
which passes through every branch at most twice, in 
opposite direction (Lemma 2.2 of \cite{H06}, see also
\cite{Mo03}).

Let $\eta$ be a large train track. If $\eta\prec\tau$ then
$\tau$ is large as well. In particular, if 
$\eta^\prime<\eta$ is a large subtrack of $\eta$ and 
if $\xi$ is a diagonal extension of $\eta^\prime$, then
a carrying map $F:\eta\to \tau$ induces a carrying
map of $\eta^\prime$ onto a large subtrack 
$\tau^\prime$ of $\tau$,
and it induces a carrying map of $\xi$ onto a
diagonal extension of $\tau^\prime$.

\begin{definition}
A pair $\eta\prec\tau$ of large train tracks
is called \emph{wide} if 
every simple closed curve which is
carried by a diagonal extension of a large subtrack
of $\eta$ fills a diagonal extension of  
a large subtrack of $\tau$. 
\end{definition}

We have

\begin{lemma}\label{widecarry}
If $\sigma\prec\eta\prec\tau$ and if the pair
$\eta\prec\tau$ is wide then $\sigma\prec\tau$ is wide.
\end{lemma}
\begin{proof} 
Let $\sigma^\prime$ be a large subtrack of 
$\sigma$ and let $\xi$ be a diagonal extension of 
$\sigma^\prime$. Then the carrying map $\sigma\to \eta$
maps $\sigma^\prime$ onto a large subtrack
$\eta^\prime$ of $\eta$, and it maps
$\xi$ to a diagonal extension $\zeta$ of $\eta^\prime$.
Similarly, $\eta^\prime$ is mapped to a large subtrack
$\tau^\prime$ of $\tau$, and $\zeta$ is mapped to 
a diagonal extension $\rho$ of $\tau^\prime$.

A simple closed curve $\alpha$ carried by $\xi$
is carried by $\zeta$. 
In particular, since
$\eta\prec\tau$ is wide, $\alpha$ fills a large subtrack of 
$\rho$. 
From this the lemma follows.
\end{proof}

A \emph{splitting and shifting sequence} 
is a finite sequence $(\tau_i)_{0\leq i\leq n}$ of large train tracks so that
for each $i$, $\tau_{i+1}$ can be obtained from $\tau_{i}$
by a sequence of \emph{shifts} followed by a single \emph{split}.
We refer to p.119 of 
\cite{PH92} and p.192 of \cite{H06} for the definition of a split and
a shift of a train track on $S$. 
We allow the split to be a \emph{collision} (see p.119 of \cite{PH92}), 
i.e. a split followed by the removal of the diagonal of the split. 
Such a collision
reduces the
number of branches of the train track. Note that $\tau_{i+1}$ is carried by $\tau_{i}$
for all $i$ and the
pair $\tau_{i+1}\prec\tau_{i}$ is \emph{never} wide. Namely, 
the cone of transverse measures for $\eta$ maps via the carrying
map onto the subcone of the cone of transverse measures on $\tau$
obtained by intersecting the latter cone with a half-space. 
This implies that there exists
an extreme ray of the cone for $\eta$ which also is an extreme ray
for the cone for 
$\tau$, and such an extreme ray is spanned by a vertex cycle $c$ for $\eta$ which 
maps to a vertex cycle of $\tau$. However, vertex cycles do not fill large
subtracks (\cite{H06}, see also \cite{Mo03}).

For an essential simple closed curve $c$ on $S$ 
let $i(c)\in \{0,\dots,n\}$ be the largest number with the
following property. There is a large subtrack $\eta$ of 
$\tau_{i(c)}$ so that $c$ is carried by a diagonal extension 
$\xi$ of $\eta$ and fills $\xi$.  If no such number exists
then put $i(c)=0$. 
 
The curve graph  ${\cal C\cal G}$ of 
$S$ is the graph whose vertices are 
simple closed curves on $S$  and where two
such curves $c,d$
are connected by an edge of length one if they can be
realized disjointly. 
Define a projection
$P:{\cal C\cal G}\to (\tau_i)_{0\leq i\leq n}$ by 
\[P(c)=\tau_{i(c)}.\] 
Extend the map $P$ to the
edges of ${\cal C\cal G}$ by mapping an edge to 
the image of one of its endpoints.

\begin{lemma}\label{project}
Let $c,d$ be disjoint simple closed curves on $S$.
Assume that $P(c)=\tau_i$. If $\tau_i\prec\tau_j$ is wide
then $P(d)=\tau_s$ for some $s\geq j$.
\end{lemma}
\begin{proof}
Assume that $P(c)=\tau_i\prec \tau_j$ is wide. By 
the definition of the map $P$, 
there is a large subtrack $\eta$ of $\tau_i$ so that
$c$ fills a diagonal extension $\xi$ of $\eta$. 
By Lemma 4.4 of \cite{MM99}, since $d$ is disjoint from $c$, 
$d$ is carried by a diagonal extension $\zeta$ of 
$\xi$. Then $\zeta$ is a diagonal
extension of $\eta$.

Since $\tau_i\prec\tau_j$ is wide, 
$d$ fills a diagonal extension of a large subtrack of 
$\tau_j$. This implies that $P(d)=\tau_s$ for some $s\geq j$.
\end{proof}

Define a distance function $d_g$ on $(\tau_i)_{0\leq i\leq n}$ as follows.
For $i<j$, the \emph{gap distance} $d_g(\tau_i,\tau_j)$
between $\tau_i$ and $\tau_j$  is the 
smallest number $k>0$ so that there is a sequence
$i_0=i<i_1<\dots< i_k=j$ with the property that for each $p<k$, 
the pair $\tau_{i_{p+1}}\prec \tau_{i_{p}}$ is \emph{not} wide.
Note that this defines indeed a distance since 
for each $\ell$ the pair $\tau_{\ell+1}\prec\tau_{\ell}$ is not wide
and hence $d_g(\tau_i,\tau_j)\leq j-i$. Moreover, 
the triangle inequality is immediate from Lemma \ref{widecarry}.

The following is a consequence of Lemma \ref{project}.
For its formulation, define a map
$P$ from a metric space $X$ to a metric space $Y$ to be
\emph{coarsely $L$-Lipschitz} for some $L>1$ if
$d(Px,Py)\leq Ld(x,y)+L$ for all $x,y\in X$.

\begin{corollary}\label{project2}
The map $P:{\cal C\cal G}\to ((\tau_i),d_g)$ is coarsely  
$2$-Lipschitz.
\end{corollary}

Define a map
$\Upsilon:(\tau_i)_{0\leq i\leq n}\to {\cal C\cal G}$ 
by associating to the train track $\tau_i$ 
one of its vertex cycles.  We have

\begin{lemma}\label{gapmetric}
The map $\Upsilon:((\tau_i),d_g)\to {\cal C\cal G}$ 
is coarsely $22$-Lipschitz. 
\end{lemma}
\begin{proof} It suffices to show the following.
If $\tau\prec \eta$ is not wide then 
the distance in ${\cal C\cal G}$ between
a vertex cycle $\alpha$ of $\tau$ and a vertex
cycle $\beta$ of $\eta$ is at most $22$.

To this end note that if 
$\alpha$ is a simple closed curve which is 
carried by a large train track $\xi$ 
then the image of $\alpha$ under a carrying map
is a subtrack of $\xi$. If this subtrack is not large 
then $\alpha$ is disjoint from an essential simple closed  
curve $\alpha^\prime$
which can be represented by an edge-path in $\xi$
(possibly with corners) which passes through 
any branch of $\xi$ at most twice. 
Since a vertex cycle of $\xi$ passes through
each branch of $\xi$ at most twice \cite{Mo03,H06}, 
this implies that $\alpha^\prime$ intersects a 
vertex cycle of $\xi$ in at most $4$ points
(Corollary 2.3 of \cite{H06}). In particular,
the distance in ${\cal C\cal G}$ between 
$\alpha$ and a vertex cycle of $\xi$ 
is at most $6$ \cite{MM99}.

On the other hand, if 
$\tau$ is another large train track and
if $\xi$ is a diagonal extension of a large subtrack 
$\tau^\prime$ of $\tau$ then 
a vertex cycle of $\xi$ intersects a vertex cycle of 
$\tau$ in at most $4$ points. Hence
the distance in ${\cal C\cal G}$ 
between a vertex cycle of $\tau$
and a vertex cycle of $\xi$ is at most $5$.
Together we deduce that 
the distance in ${\cal C\cal G}$ between
$\alpha$ and a vertex cycle of $\tau$ does not exceed $11$.

Now by definition, if $\tau\prec \eta$ is not wide then there is 
a curve $\alpha$ which is 
carried by a diagonal extension 
$\xi$ of a large subtrack $\tau^\prime$ of $\tau$
and such that the following holds true.
A carrying map $\tau\to \eta$ induces a carrying map of $\xi$ onto 
a diagonal extension $\zeta$ of a large subtrack of $\eta$.
The train track $\zeta$ carries 
$\alpha$ and so that $\alpha$ does not fill a large
subtrack of $\zeta$.  Since a carrying map 
$\xi\to \zeta$ maps a large subtrack of $\xi$ onto
a large subtrack of $\zeta$, the curve $\alpha$ does not
fill a large subtrack of $\xi$.

By the above discussion, the distance in ${\cal C\cal G}$
between $\alpha$ and any vertex cycle of 
both $\tau$ and $\eta$ is at most $11$.
This shows the lemma.
\end{proof}

Call a map $\Phi$ of a metric space
$(X,d)$ into a subset
$A$ of $X$ 
a \emph{coarse Lipschitz 
retraction} if there is a number $L>1$ with the following
properties.
\begin{enumerate}
\item $d(\Phi(x),\Phi(y))\leq Ld(x,y)+L.$
\item $d(x,\Phi(x))\leq L$ whenever $x\in A$.
\end{enumerate}

We are now ready to show

\begin{corollary}\label{retract}
For any splitting and shifting sequence $(\tau_i)_{0\leq i\leq n}$ 
the map $\Upsilon \circ P$ is a coarse $L$-Lipschitz
retraction of ${\cal C\cal G}$ for a number $L>1$ not depending on $(\tau_i)$
or on the Euler characteristic  of $S$.
\end{corollary}
\begin{proof} Let $d_{\cal C\cal G}$ be the distance in the curve graph of $S$.
By Corollary \ref{project2} and 
Lemma \ref{gapmetric}
it suffices to show that
$d_{\cal C\cal G}(\alpha,\Upsilon\circ P(\alpha))\leq L$ for a universal constant
$L>1$ and every vertex cycle $\alpha$ of  a train track $\tau_i$
from the sequence.

To this end observe that since $\alpha$ is a vertex cycle of $\tau_i$,
$\alpha$ is carried by each of the train tracks $\tau_j$ for $j\leq i$, 
moreover $\alpha$ does not fill a diagonal extension of a large
subtrack of $\tau_i$. On the other hand,
by definition of a wide pair, 
if $\tau_i\prec \tau_j$ is wide then 
$\alpha$ fills a large subtrack of $\tau_j$. This means
that $P(\alpha)=\tau_s$ for some $s\geq j$ so that
the pair $\tau_{i}\prec \tau_{s+1}$ is \emph{not} wide.
The corollary now follows from Lemma \ref{gapmetric}. 
\end{proof}

{\bf Remark:} The above discussion immediately implies 
that the image under $\Upsilon$ of a splitting and shifting
sequence of train tracks is an unparametrized quasi-geodesic
in ${\cal C\cal G}$ for a constant not depending on the Euler
characteristic of $S$. A non-effective version of this result  
was earlier established in \cite{MM04}
(see also \cite{H06}).

\section{Quasi-geodesics in the superconducting disk graph}
\label{quasigeodesics}

In this section we resume the discussion of an oriented 3-manifold
$M$ and a subsurface $X$ of the boundary of $M$ whose boundary
is incompressible in $M$.

Recall the definition of the graph
${\cal S\cal D\cal G}(X)$. 
Our goal is to show that
the natural map which associates to a disk its boundary
defines a quasi-isometric embedding of 
${\cal S\cal D\cal G}(X)$ into the curve graph 
${\cal C\cal G}(X)$ of $X$. To simplify the notation we
identify in the sequel a disk in $M$ with boundary in $X$ 
with its boundary circle.
Thus we view the vertex set of ${\cal S\cal D\cal G}(X)$ as
a subset of the curve graph ${\cal C\cal G}(X)$ of $X$.

The argument is based on the results in Section 2-3
and a construction from \cite{MM04}.
This construction uses a specific type of surgery sequences 
of disks which can be related to train tracks
as follows.

Let  $D,E\subset H$ be two disks in normal position, with boundary
in $X$. Let $E^\prime$
be an outer component of $E-D$ and let $D_1$ be a disk
obtained from $D$ by simple surgery at $E^\prime$. 

Let $\alpha$ be the intersection of $\partial E^\prime$ with
$X$. Then up to isotopy, the boundary $\partial D_1$ of the
disk $D_1$ contains $\alpha$ as an embedded subarc. Moreover,
$\alpha$ is disjoint from $E$. In particular, given an outer component
$E^{\prime\prime}$ of $E-D_1$, there is a distinguished choice for
a disk
$D_2$ obtained from $D_1$ by simple surgery at $E^{\prime\prime}$.
 The disk $D_2$  
is determined by the requirement that $\alpha$ is \emph{not}
a subarc of $\partial D_2$.  
Then for an outer component of 
$E-D_2$ there is a distinguished choice for a disk $D_3$ obtained
from $D_2$ by simple surgery at an outer component
of $E-D_2$  etc. We call a surgery sequence $(D_i)$ of this form
a \emph{nested} surgery path in direction of $E$.
Note that the boundary of each disk $D_i$ is composed of
a single subarc of $\partial D$ and a single subarc of 
$\partial E$.

The following result is due to Masur
and Minsky (this is Lemma 4.2 of \cite{MM04} which
is based on Lemma 4.1 and the proof of Theorem 1.2 in that paper).

\begin{proposition}\label{nestedcarried}
Let 
$D,E\subset X$ be any disks. Let  
$D=D_0,\dots,D_n$ be a nested surgery path in the direction of $E$
which  
connects $D$ to a disk $D_n$ disjoint from $E$.
Then for each
$i\leq n$ there is a train track $\tau_i$ on $X$ 
with a single switch such that
the following holds true. 
\begin{enumerate}
\item $\tau_i$ carries $\partial E$ and $\partial E$ fills up $\tau_i$.
\item $\tau_{i+1}\prec \tau_{i}$.
\item The disk $D_i$ intersects $\tau_i$ only at the switch.
\end{enumerate}  
\end{proposition}

The train tracks $\tau_i$ in the proposition are constructed
as follows.

Let $\alpha=\partial D,\beta=\partial E$. Assume that the
curves $\alpha,\beta$ are smooth (for a smooth structure on
$X$) 
and fill up $X$. This means that the
complementary components of $\alpha\cup \beta$ are all
polygons or once holed polygons where in our setting, 
a hole is a boundary component of $X$.
Let $P$ be a complementary polygon
which has at least 6 sides. Such a polygon
exists since the Euler characteristic of $X$ is negative.
Its edges are subsegments of $\alpha$ and $\beta$. 
Let $I$ be a boundary edge of $P$ 
contained in $\alpha$.  
Collapse $\alpha-I$ to a single point with a homotopy $F$ of $X$.
This can be done in such a way that the restriction of $F$ to 
$\beta$ is nonsingular everywhere. The 
resulting graph has a single vertex. Collapsing the bigons in the
graph to single arcs yields a train track $\tau$ with a single
switch \cite{MM04}. 

Let $b\subset \beta$  be an outer arc for $E-D$
and let $a\subset \alpha-I$ be the subarc of $\alpha$
which is bounded by the endpoints of $b$ and which 
does not intersect the interval $I$. Then $a\cup b$ is the boundary of 
a disk $D_1$ obtained from $D$ by nested surgery at $b$.
The new train track $\tau_1$ obtained from the
above construction is
obtained from $\beta\cup a$
by collapsing the arc $a$ to a single point
(we refer to  \cite{MM04} for details).

In the formulation of the following result, $\chi(X)$ denotes the Euler
characteristic of the surface $X$.

\begin{theorem}\label{quasi}
There is an explict quadratic polynomial $p$ 
such that
the vertex inclusion defines a $p(\vert \chi(X)\vert)$-quasi-isometric embedding
${\cal S\cal D\cal G}(X)\to {\cal C\cal G}(X)$. In particular,
${\cal S\cal D\cal G}(X)$ is a hyperbolic geodesic metric graph. 
\end{theorem}
\begin{proof} As before, let $d_{{\cal S}}$ be the distance in
${\cal S\cal D\cal G}(X)$ and let $d_{{\cal C\cal G}}$ be the distance
in ${\cal C\cal G}(X)$. 
We have to show the existence of a quadratic polynomial 
$p$ with 
the following property. If $D,E$ are any disks 
then
\[d_{{\cal S}}(D,E)\leq p(\vert \chi(X)\vert)
d_{{\cal C\cal G}}(\partial D,\partial E).\]

By Proposition \ref{nestedcarried},
there is a nested surgery path $D=D_0,\dots,D_n$ connecting
the disk $D_0=D$ to a disk $D_n$ which 
is disjoint from $E$, 
and there is a sequence $(\tau_i)_{0\leq i\leq n}$ 
of one-switch train tracks on $X$
such that $\tau_{i+1}\prec\tau_{i}$ for all $i<n$ and that
$D_i$ intersects $\tau_i$ only at the switch.

By Theorem 2.3.1 of 
\cite{PH92}, there is a splitting and shifting sequence 
$\tau_0=\eta_0\prec\dots\prec\eta_s=\tau_n$ 
connecting $\tau_0$ to $\tau_n$ and 
a sequence $0=u_0<\dots <u_n=s$ so that
$\eta_{u_q}=\tau_q$ for  $0\leq q\leq n$. 
Since the disk $D_i$ intersects $\tau_i$ only at the switch,
the boundary $\partial D_i$ of $D_i$ intersects
a vertex cycle of $\tau_i$ in at most two points and hence
the distance in the curve graph ${\cal C\cal G}(X)$ 
between $\partial D_i$ and a vertex cycle of 
$\tau_i$ is at most three. 
Now the disks $D_i$ and $D_{i+1}$ are 
disjoint and  consequently 
the distance in ${\cal C\cal G}(X)$ between a vertex
cycle of $\tau_i$ and a vertex cycle of $\tau_{i+1}$ is 
at most $7$. 

Corollary \ref{retract} implies that the map $\Upsilon$ which associates
to the train track $\eta_u$ one of its vertex cycles is 
a quasi-isometric embedding of the splitting and shifting 
sequence $(\eta_u)$, equipped with the gap distance $d_g$, 
into the curve graph of $X$. The discussion in the previous
paragraph implies that this statement also holds true for 
the restriction of $\Upsilon$ to the subsequence $(\tau_i)$ of 
$(\eta_u)$, equipped with the restriction of the gap distance. 
Thus 
by the definition of the gap distance,
it suffices to show the existence of a universal 
number $b>0$ with the
following property.
Let $k<i$ be such that
the pair $\tau_i\prec\tau_k$ is \emph{not} wide;
then $d_{{\cal S}}(D_{i},D_{k})\leq b\vert \chi(X)\vert^2$.

Since $\tau_i\prec\tau_k$ is not wide
there is a large subtrack $\tau_i^\prime$
of $\tau_i$, a diagonal extension $\zeta_i$ of $\tau_i^\prime$ and
a simple closed curve $\alpha$ carried
by $\zeta_i$ with the following property.
Let $\tau_k^\prime$ be the image of 
$\tau_i^\prime$ under a carrying map 
$\tau_i\to \tau_k$ and let
$\zeta_k$ be the diagonal extension of $\tau_k^\prime$
which is the image of $\zeta_i$ under a carrying
map induced by 
a carrying map $\tau_i^\prime\to 
\tau_k^\prime$. Then $\alpha$ does not fill 
a large subtrack of $\zeta_k$.

Since $\zeta_i$ is a diagonal extension of the
large subtrack $\tau_i^\prime$ 
of $\tau_i$ and since $D_i$ intersects
$\tau_i$ only at the switch,
the intersection number
between $\partial D_i$ and $\zeta_i$ is bounded from 
above by a constant $\kappa\geq 2$
which  does not exceed a constant multiple of the 
Euler characteristic of $X$.

For each $p\in [k,i]$, the image
of $\tau_i^\prime$ under a carrying map
$\tau_i\to \tau_p$ is 
a large subtrack $\tau_p^\prime$ 
of $\tau_p$, and
there is a diagonal extension $\zeta_p$ of 
$\tau_p^\prime$ which carries $\alpha$. 
We may assume that $\zeta_u\prec\zeta_p$ for $u\geq p$.
The disk $D_p$ intersects
$\zeta_p$ in at most $\kappa$ points.

For $p\in [k,i]$ let $\beta_p\prec\zeta_p$ be the
subtrack of $\zeta_p$ filled by $\alpha$. Then $\beta_p$ is 
connected and not large.
The union $Y_p$ of a thickening of 
$\beta_p$ with the 
components of $X-\beta_p$ which are simply connected
is a proper connected subsurface
of $X$ for all $p$.
The boundary of $Y_p$ can be realized as a 
union of simple closed curves which are
embedded in $\beta_p$ (but with cusps). 
The carrying map 
$\beta_{p+1}\to \beta_p$ maps
$Y_{p+1}$ into $Y_p$. In particular, either the
boundary of $Y_{p+1}$ coincides up to homotopy 
with the boundary of 
$Y_p$ or $Y_{p+1}$ is a \emph{proper} subsurface of 
$Y_p$. In the latter case, the Euler characteristic of 
$Y_p$ is strictly smaller than the Euler characteristic of 
$Y_{p+1}$. 
In other words, the subsurfaces
$Y_p$ are nested, and hence their number is bounded 
from above by a universal constant $h >0$ depending 
linearly on the Euler characteristic of $S$.

Since $\partial D_p$ intersects $\zeta_p$ in 
at most $\kappa$ points, the number of intersections between
$\partial D_p$ and $\partial Y_{p}$ 
is bounded from
above by $2\kappa >0$. 
As a consequence, 
there are $h$ essential simple closed curves
$c_1,\dots,c_{h}$ in $X$ 
so that for every 
$p\in[k,i]$ there is some $r(p)\in \{1,\dots,h\}$
with 
\begin{equation}\label{diskone}
\iota(\partial D_p,c_{r(p)})\leq 2\kappa.\end{equation}
Each of the curves $c_j$ is a fixed boundary component
of one of the subsurfaces $Y_p$.

By reordering, 
assume that $r(i)=1$. Let $v_1$ be the minimum of
all numbers $p\in [k,i]$ 
such that $r(v_1)=1$. 
Proposition \ref{distanceinter2}
shows that 
\begin{equation}\label{disktwo}
d_{{\cal S}}(D_{i},D_{v_1})\leq 4\kappa+4.\end{equation}
On the other hand, we have $d_{{\cal S}}(D_{v_1},D_{v_1-1})=1$.
Again by reordering, assume 
that $r(v_1-1)=2$ and repeat this construction
with the disks $D_{v_1-1},\dots,D_{k}$ and 
the curve $c_2$. 
In $a\leq h$ steps we 
construct in this way a decreasing sequence
$i\geq v_1>\dots>v_a=k$ such that
$d_{{\cal S}}(D_{v_u},D_{v_{u-1}})\leq 4\kappa+5$ for all $u\leq a$.
From (\ref{diskone}, \ref{disktwo}) we conclude that 
\begin{equation}\label{diskthree}
d_{{\cal S}}(D_{i},D_{k})\leq h(4\kappa+5).\notag\end{equation}
Together with the explicit bounds for 
$\kappa$ and $h$, this yields the theorem.
\end{proof}

\section{Gromov boundary in the case of handlebodies}\label{gromov}

A hyperbolic geodesic metric space $Y$ admits a \emph{Gromov boundary}.
This boundary is a topological space on which the
isometry group of $Y$ acts as a group of homeomorphisms.
In this section we explicitly determine 
the Gromov
boundary of the superconducting disk 
graph ${\cal S\cal D\cal G}={\cal S\cal D\cal G}(\partial H)$
for a handlebody $H$ of genus $g\geq 2$. Recall that the boundary
$\partial H$ of such a handlebody $H$ is a closed oriented
surface of genus $g$.

Let ${\cal L}$ be the space
of all \emph{geodesic laminations} on 
$\partial H$ (for some fixed hyperbolic metric)
equipped with the
\emph{coarse Hausdorff topology}. In this topology,
a sequence $(\mu_i)$ converges to a lamination
$\mu$ if every accumulation point of $(\mu_i)$ 
in the usual Hausdorff topology contains $\mu$ as a 
sublamination. Note that the coarse Hausdorff topology on 
${\cal L}$ is not
$T_0$, but its restriction to the subspace 
$\partial {\cal C\cal G}\subset {\cal L}$ of all minimal geodesic  
laminations which fill up $\partial H$ 
(i.e. which intersect every simple closed geodesic transversely)
is Hausdorff. 
The space $\partial {\cal C\cal G}$ equipped with
the coarse Hausdorff topology can naturally be identified
with the Gromov boundary  of 
the curve graph ${\cal C\cal G}$ of $\partial H$ \cite{K99,H06}.

Let ${\rm Map}(H)$ be the \emph{handlebody group}, which 
is defined to be the subgroup of the mapping class
group ${\rm Mod}(\partial H)$ of the boundary surface consisting
of all isotopy classes of diffeomorphisms which extend to 
diffeomorphisms of $H$. The group ${\rm Map}(H)$ acts on the
graph ${\cal S\cal D\cal G}$ as a group of simplicial automorphisms.

The handlebody group ${\rm Map}(H)$ also
acts on $\partial{\cal C\cal G}$ as a group of transformations
preserving the 
closed subset 
\[\partial{\cal H}\subset \partial{\cal C\cal G}\] 
of all geodesic 
laminations which are limits in the coarse Hausdorff topology
of boundaries of disks in $H$. 
It acts on the Gromov boundary $\partial {\cal S\cal D\cal G}$
of ${\cal S\cal D\cal G}$ as well.

\begin{lemma}\label{closedsubset}
The Gromov boundary of ${\cal S\cal D\cal G}$ is a 
closed ${\rm Map}(H)$-invariant subset of 
$\partial{\cal H}$.
\end{lemma}
\begin{proof} Since by Theorem \ref{quasi}
the vertex inclusion ${\cal S\cal D\cal G}\to
{\cal C\cal G}$ defines a quasi-isometric embedding,
the Gromov boundary of ${\cal S\cal D\cal G}$ is the subset of
the Gromov boundary of ${\cal C\cal G}$ of all endpoints of
quasi-geodesic rays in ${\cal C\cal G}$ 
which are contained in 
${\cal S\cal D\cal G}$. 

By the main result of 
\cite{H06} (see \cite{K99} for an earlier account of a similar statement),
a simplicial quasi-geodesic ray $\gamma:[0,\infty)\to{\cal C\cal G}$
defines the endpoint lamination $\nu\in \partial{\cal C\cal G}$ if and
only if the curves $\gamma(i)$ converge as 
$i\to \infty$ in the coarse Hausdorff topology to
$\nu$. As a consequence, the Gromov boundary 
of ${\cal S\cal D\cal G}$ is a subset of 
$\partial{\cal H}$, and this subset is 
clearly ${\rm Map}(H)$-invariant.

We are left with showing that the Gromov boundary of ${\cal S\cal D\cal G}$
 is a closed subset of $\partial {\cal C\cal G}$. To this end note
 that by Theorem \ref{quasi}, 
there is a number $p>1$ such that 
 for every $L>1$, any $L$-quasi-geodesic in  
 ${\cal S\cal D\cal G}$ is an $Lp$-quasi-geodesic in ${\cal C\cal G}$.
 Moreover, for a suitable choice
 of $p$, any vertex in ${\cal S\cal D\cal G}$ can be connected
 to any point in the Gromov boundary of ${\cal S\cal D\cal G}$ by
 a $p$-quasi-geodesic.
 
Now let $(\nu_i)$ be a sequence
in the Gromov boundary of ${\cal S\cal D\cal G}$ 
which converges in $\partial {\cal C\cal G}$ 
to a lamination $\nu$.
Let $\partial D$ be the boundary of a disk and let 
$\gamma:[0,\infty)\to {\cal C\cal G}$ be 
a quasi-geodesic ray issuing from
$\gamma(0)=\partial D$ with endpoint $\nu$. 
By hyperbolicity of ${\cal C\cal G}$ and by the
discussion in the previous paragraph, 
there is a number $R>0$ and 
for every $k\geq 0$ there is some $i(k)>0$ such that
a $p$-quasi-geodesic in 
${\cal S\cal D\cal G}\subset {\cal C\cal G}$ connecting
$\gamma(0)$ to $\nu_{i(k)}$ passes through the
$R$-neighborhood of $\gamma(k)$
in ${\cal C\cal G}$. Since $k>0$ was arbitrary,
this implies that the entire quasi-geodesic ray $\gamma$ is contained
in the $R$-neighborhood of the subset ${\cal S\cal D\cal G}$ of 
${\cal C\cal G}$. Using once more hyperbolicity, we conclude that
there is a quasi-geodesic ray in ${\cal C\cal G}$ 
connecting $\gamma(0)$ to $\nu$ which is entirely 
contained in 
${\cal S\cal D\cal G}$. But this just means that $\nu$ is 
contained in the Gromov boundary of ${\cal S\cal D\cal G}$.
\end{proof}

By naturality, the action of 
the handlebody group ${\rm Map}(H)$ on the
Gromov boundary $\partial {\cal S\cal D\cal G}$ 
of ${\cal S\cal D\cal G}$ 
is compatible
with the action of the mapping class group 
on the Gromov boundary of the curve
graph. From Lemma \ref{closedsubset}
and the following observation
(which is essentially contained in 
Theorem 1.2 of \cite{M86}), we conclude that
$\partial{\cal H}$ is indeed the Gromov boundary of 
${\cal S\cal D\cal G}$.

\begin{lemma}\label{minimal}
The action of the handlebody group ${\rm Map}(H)$ on 
$\partial {\cal S\cal D\cal G}$ is minimal. 
\end{lemma}
\begin{proof} Let 
$(\partial D_i)$ be a sequence of boundaries of 
disks $D_i$ converging in the 
coarse Hausdorff topology to a geodesic lamination
$\mu\in \partial{\cal H}$. For each 
$i$ let $E_i$ be a disk which is disjoint from $D_i$. 
Since the space of geodesic laminations equipped with the
usual Hausdorff topology is compact, 
up to passing to a subsequence
the sequence $(\partial E_i)$ converges in the Hausdorff topology
to a geodesic lamination $\nu$ which does not intersect $\mu$
(we refer to \cite{K99,H06} for details of this argument).
Now
$\mu$ is minimal and 
fills up $\partial H$ and therefore  the lamination $\nu$ contains
$\mu$ as a sublamination. This just  means that $(\partial E_i)$
converges in the coarse  Hausdorff
topology to $\mu$.

Since the genus of $H$ is at least two, 
for every separating 
disk in $H$ we can find a disjoint non-separating disk.
Thus the discussion in the previous paragraph 
shows that  
every $\mu\in \partial{\cal H}$ is a limit 
in the coarse Hausdorff topology 
of a sequence of
non-separating disks. However, the handlebody group 
acts transitively on non-separating disks.
Minimality of the action of ${\rm Map}(H)$ on 
$\partial{\cal S\cal D\cal G}$ follows.
\end{proof}

As an immediate consequence of Lemma \ref{closedsubset}
and Lemma \ref{minimal} we obtain

\begin{corollary}\label{gromovsuper}
$\partial {\cal H}$ 
is the Gromov boundary of ${\cal S\cal D\cal G}$.
\end{corollary}

 \section{Disk graphs for surfaces with 
 compressible boundary}\label{disksforahandle}

In this final section we show that the results from Sections 2-4 
are not valid for graphs of disks in an oriented 3-manifold $M$
with boundary in a compact 
subsurface $X$ of the boundary of $M$ of genus $g\geq 2$, with 
connected 
compressible boundary. 
As before, disks are required to be essential in $M$, and their
boundaries are required to be essential curves in $X$, in 
particular their boundaries are not allowed to be homotopic to
a boundary component of $X$. 
We continue to use the terminology from Section 2.
Theorem \ref{spot} from the introduction.

Let $X_0$ be obtained from
$X$ by capping off the boundary $\partial X$ 
(i.e. identify the boundary $\partial X$ with a single point). 
Note that 
$X_0$ can be viewed as 
a subsurface of the boundary of a submanifold $M_0$ of $M$
with boundary. There exists a natural map $\Phi:X\to X_0$.  

Let ${\cal C\cal G}(X)$ be the curve graph of $X$
and 
let ${\cal C\cal G}(X_0)$ be the curve graph of $X_0$.
The following simple and well known fact is the essential 
feature that distinguishes the case of a single compressible
boundary component from the case of more than one
compressible boundary components.

\begin{lemma}\label{treebundle} 
The map $\Phi$ induces a simplicial surjection 
\[\Pi:{\cal C\cal G}(X)\to {\cal C\cal G}(X_0)\]
which maps diskbounding curves to diskbounding curves.
\end{lemma}
\begin{proof} Since $\partial X$ has connected compressible boundary,
the image under the map $\Phi$ of an essential simple closed 
curve $\gamma$ 
on $X$ is an essential simple closed
curve $\Phi(\gamma)$ on $X_0$. The curve 
$\gamma$ is diskbounding if and only if this is the
case for $\Phi(\gamma)$. Moreover, if $\gamma,\delta$ are disjoint
then this holds true for $\Phi(\gamma),\Phi(\delta)$ as well.
This immediately implies the lemma.
\end{proof}

The special property of a boundary surface with connected compressible
boundary 
which enters Lemma \ref{treebundle} 
is also reflected in the fact that we can use surgery of disks
to construct paths in the disk graph which reduce distances.
Namely, for any two disks $D,E$ 
with boundary in $X$ 
which are not disjoint
and for any outer component $E^\prime$ of $E-D$, at
least one of the disks obtained from $D$ by surgery at 
$E^\prime$ is not peripheral. Thus if we
denote as before by $d_{\cal D}$ and $d_{\cal E}$ the 
distance in the disk graph and the electrified disk graph,
then the proof of
Lemma \ref{fellowtravel} yields the following

\begin{lemma}\label{firstestimate0}
If $X$ has connected compressible boundary 
then for any disks  
$D,E$ in $M$ we have 
 \[d_{{\cal D}}(D,E) \leq \iota(\partial D,\partial E)/2+1.\] 
\end{lemma}

The difficulty that only one of the two 
possible choices for simple surgery may be essential
is reflected in the following

\begin{proposition}\label{notquasicon}
The graph of disks with boundary in $X$ 
is not a quasi-convex subset of the 
curve graph of $X$.
\end{proposition}
\begin{proof} The curve graph ${\cal C\cal G}(X)$ 
of $X$ is hyperbolic,
and pseudo-Anosov elements of the mapping class
group of $X$ act as hyperbolic isometries. 

Let $p$ be the image of the boundary of $X$ under the 
map $\Phi:X\to X_0$. 
View the point  $p$ in as a basepoint for
the fundamental group of $X_0$. 
Denote by $\Gamma$ the quotient of the mapping class
group of $X$ by its center, which is just the mapping
class group of a surface of genus $g$ with one marked point. 
In the \emph{Birman exact sequence}
\begin{equation}\label{birman}
0\to \pi_1(\partial H_0,p)\to \Gamma \to 
{\rm Mod}(X_0)\to 0,\end{equation}
an element $\gamma\in \pi_1(X_0,p)$ 
is mapped to a so-called \emph{point-pushing map}
$\Psi(\gamma)\in \Gamma$.
If $\gamma\in \pi_1(X_0,p)$
is \emph{filling}, i.e. if $\gamma$ decomposes $X_0$
into disks, then the image $\Psi(\gamma)$ of $\gamma$ in
$\Gamma$ via the Birman exact sequence is 
pseudo-Anosov \cite{Kr81,KLS09}.

Let $\phi$ be a diffeomorphism of $X_0$
which fixes the basepoint $p$ and which 
defines a pseudo-Anosov element of 
${\rm Mod}(X_0)$. We require that a quasi-axis $\zeta$ for
the action of $\phi$ on ${\cal C\cal G}(X_0)$  
passes uniformly near the boundary 
of a disk and that moreover  
for any diskbounding simple closed
curve $\zeta$, the distance in ${\cal C\cal G}(X_0)$ between
$\phi^k\zeta$ and the quasi-convex 
subset of diskbounding curves tends to infinity
as $k\to\infty$. Note that the results from Section 4 apply
to the graph of disks in the 3-manifold $M_0$ with boundary
in $X_0$. 

Such a pseudo-Anosov element can be found as follows.
Each pseudo-Anosov element fixes two projective 
measured laminations which fill up $X_0$. This  
means that the complementary
components of the lamination are all simply connected. 
The set of pairs of such fixed points is dense in
${\cal P\cal M\cal L}\times {\cal P\cal M\cal L}$ 
(here ${\cal P\cal M\cal L}$ denote the Thurston sphere of 
projective measured geodesic laminations on $X_0$).
The closure in ${\cal P\cal M\cal L}$
of the set of diskbounding simple closed curves is nowhere
dense in ${\cal P\cal M\cal L}$ (see \cite{M86} for details and note
that by standard 3-dimensional topology, the image of 
the fundamental group of $X_0$ in the fundamental group of $M_0$
can not be trivial as the genus of $X_0$ is at least two). 
Any  pseudo-Anosov element
$\phi$ whose pair of fixed points is contained in the complement
will do.

Since $\phi$ fixes the point $p$,
it acts on the fundamental group $\pi_1(X_0,p)$
of $X_0$, 
moreover it can be viewed as an 
element of $\Gamma$. We denote this element
of $\Gamma$ again by $\phi$.

The idea is now to conjugate the point-pushing map $\Psi(\gamma)$
by high powers $\phi^k$ of $\phi$.
The resulting mapping class is pseudo-Anosov, and 
$\phi^k\zeta$ is a quasi-axis for its action on 
${\cal C\cal G}(X_0)$, with all constants uniform in $k$.
As $\phi$ acts with north-south dynamics
on ${\cal P\cal M\cal L}$ 
and on the Gromov boundary of the curve graph
of $X_0$, by hyperbolicity of
${\cal C\cal G}(X_0)$, for large enough $k$ the 
quasi-axis $\phi^k\zeta$ for the action of 
$\phi^k\circ \Psi(\gamma)\circ \phi^{-k}$  
is arbitrarily far from the quasi-convex disk set.

Now let $\beta\subset X_0-\{p\}$ 
be a diskbounding curve near a quasi-axis of $\phi$. Since $\beta$
avoids $p$ we can view $\beta$ as a diskbounding curve in 
$X$. 
For each $k>0$ let $\beta_k$ be the image of 
$\beta$ under point-pushing along the 
curve $\phi^{k}(\gamma)$. Then $\beta_k$ is diskbounding,
moreover we have
\[\beta_k=(\phi^k\circ \Psi(\gamma)\circ \phi^{-k})(\beta).\]

By hyperbolicity of ${\cal C\cal G}(X)$, via
perhaps replacing $\gamma$ by a multiple 
(and hence replacing the point pushing map $\Psi(\gamma)$ by some power)
we may assume that a geodesic 
in ${\cal C\cal G}(X)$ connecting 
$\beta$ to $\beta_k$ is close in the Hausdorff topology to 
the composition of three arcs. The first arc connects
$\beta$ to the quasi-axis 
$\phi^k(\zeta)$ of $\phi^k\circ \Psi(\gamma)\circ \phi^{-k}$, 
the second arc travels along $\phi^k(\zeta)$, and the 
third arc connects $\phi^k(\zeta)$
to $\beta_k$. 

However, by the choice of $\phi$, for suitable choices of $k$ and 
suitable multiplicities of $\gamma$, such a curve
is arbitrarly
far in ${\cal C\cal G}(X)$ 
from the set of diskbounding curves.
\end{proof}

\bigskip\bigskip

\noindent
MATH. INSTITUT DER UNIVERSIT\"AT BONN\\
ENDENICHER ALLEE 60\\
53115 BONN\\
GERMANY\\

\bigskip\noindent
e-mail: ursula@math.uni-bonn.de

\end{document}